\documentclass[]{siamart1116}

\usepackage{amsfonts}
\ifpdf
  \DeclareGraphicsExtensions{.eps,.pdf,.png,.jpg}
\else
  \DeclareGraphicsExtensions{.eps}
\fi
\numberwithin{theorem}{section}

\usepackage{amsmath}
\usepackage{amssymb}
\usepackage{mathtools}
\usepackage{bbm}
\usepackage{graphicx}
\usepackage{epstopdf}
\usepackage{multirow}
\usepackage{wrapfig}
\usepackage{csquotes}
\usepackage{upquote}
\usepackage{mathrsfs}
\usepackage{enumitem}
\usepackage{verbatim}
\usepackage{subcaption}

\usepackage{algorithm}
\usepackage{algorithmicx}
\usepackage[noend]{algpseudocode}

\usepackage{hyperref}

\usepackage{tikz}
\usepackage{pgflibraryshapes}
\usetikzlibrary{decorations.text}
\usetikzlibrary{decorations.pathreplacing}
\usetikzlibrary{backgrounds}
\usetikzlibrary{calc}
\usetikzlibrary{arrows}
\usetikzlibrary{snakes}
\usetikzlibrary{positioning}
\definecolor{cffffff}{RGB}{255,255,255}
\definecolor{cff00c0}{RGB}{255,0,192}
\definecolor{c0f00ff}{RGB}{15,0,255}
\definecolor{c00ff20}{RGB}{0,255,32}
\definecolor{cfffd00}{RGB}{255,253,0}
\definecolor{c00e0ff}{RGB}{0,224,255}
\usepackage{pgfplots}
\usepackage{pgfplotstable}

\usepackage{color}
\definecolor{deepblue}{rgb}{0,0,0.5}
\definecolor{deepred}{rgb}{0.6,0,0}
\definecolor{deepgreen}{rgb}{0,0.5,0}

\DeclareFixedFont{\ttb}{T1}{cmtt}{bx}{n}{10} 
\DeclareFixedFont{\ttm}{T1}{cmtt}{m}{n}{10}  
\DeclareFixedFont{\sttb}{T1}{cmtt}{bx}{n}{8} 
\DeclareFixedFont{\sttm}{T1}{cmtt}{m}{n}{8}  
\DeclareMathAlphabet{\mathpzc}{OT1}{pzc}{m}{it}

\newcommand{\R}{\mathbb{R}}

\newcommand{\Q}{\mathbb{Q}}

\newcommand{\I}{\mathbb{I}}
\newcommand{\N}{\mathbb{N}}

\newcommand{\curly}[1]{\mathscr{#1}}
\newcommand{\norm}[1]{\left\lVert#1\right\rVert}
\newcommand{\tab}{\hspace{1cm}}

\newtheorem{claim}[theorem]{Claim}

\usepackage{setspace}
\graphicspath{{../images/}}

\title{An Algorithm for Computing Lipschitz Inner Functions in Kolmogorov's Superposition Theorem %
\thanks{Submitted to the editors \today. 
}}
\author{Jonas Actor %
\thanks{Department of Computational and Applied Mathematics, Rice University, Houston, TX 77005 (\email{jonasactor@rice.edu}).} %
\and Matthew G. Knepley %
\thanks{Department of Computer Science and Engineering, University of Buffalo, Buffalo, NY 14260 (\email{knepley@buffalo.edu}).}}

\headers{Computing Lipschitz Inner Functions for KST}
{J. Actor and M. Knepley}

\begin{document}
\maketitle

\begin{abstract}

Kolmogorov famously proved in~\cite{kolmogorov} that multivariate continuous functions can be represented as a
superposition of a small number of univariate continuous functions, $$ f(x_1,\dots,x_n) = \sum_{q=0}^{2n+1} \chi^q
\left( \sum_{p=1}^n \psi^{pq}(x_p) \right).$$ Fridman~\cite{fridman} posed the best smoothness bound for the functions
$\psi^{pq}$, that such functions can be constructed to be Lipschitz continuous with constant 1. Previous algorithms to
describe these inner functions have only been H\"older continuous, such as those proposed by K\"oppen in~\cite{koppen}
and Braun and Griebel in~\cite{braun-griebel}. This is problematic, as pointed out by Griebel~\cite{griebel-highdim}, in
that non-smooth functions have very high storage/evaluation complexity, and this makes Kolmogorov's representation (KR)
impractical using the standard definition of the inner functions.

To date, no one has presented a method to compute a Lipschitz continuous inner function. In this paper, we revisit
Kolmogorov's theorem along with Fridman's result. We examine a simple Lipschitz function which appear to satisfy the
necessary criteria for Kolmogorov's representation, but fails in the limit. We then present a full solution to the
problem, including an algorithm that computes such a Lipschitz function.
\end{abstract}



\section{Kolmogorov's Superposition Theorem}

Kolmogorov proved the following theorem in 1957.

\begin{theorem}[KST \cite{kolmogorov}] \label{theorem:KST}
  Let $f: \R^n \rightarrow \R \in C\left([0,1]^n \right)$ where $n \ge 2$. Then, there exist $\psi^{pq}: [0,1] \rightarrow [0,1] \in C[0,1]$ and $\chi_q: \R \rightarrow \R \in C(\R)$,
   where $p \in \{ 1, \dots ,n \} $ and $q \in \{ 0,\dots,2n \} $, such that
  \begin{displaymath} f(x_1,\dots,x_n) = \sum_{q=0}^{2n} \chi^q \left( \sum_{p=1}^n \psi^{pq}(x_p) \right). \end{displaymath}
\end{theorem}

We could add to this statement that the ``inner'' functions $\psi^{pq}$ are independent of choice of function $f$. These inner functions can be chosen to be Lipschitz continuous with a Lipschitz constant of 1 \cite{fridman}. Sprecher in \cite{sprecher1972} reformulates this theorem by replacing the functions $\psi^{pq}$ with translations and scaling of a single function $\psi$, which can still be chosen to be Lipschitz continuous. In this formulation, theorem \ref{theorem:KST} becomes

\begin{theorem}{Sprecher's KST Reformulation \cite{sprecher1972}}
  Let $f: \R^n \rightarrow \R \in C\left([0,1]^n \right)$ where $n \ge 2$. Fix $\epsilon \le \frac{1}{2n}$, and choose $\lambda \in \R$ such that $1 = \lambda^0, \lambda^1, \dots, \lambda^{n-1}$
  are integrally independent. Then, there exist $\psi: [-1,1] \rightarrow [0,1] \in C[-1,1]$ and $\chi_q: \R \rightarrow \R \in C(\R)$ for $q \in \{0,\dots,2n\}$, such that
  $$ f(x_1,\dots,x_n) = \sum_{q=0}^{2n} \chi^q \left( \sum_{p=1}^n \lambda^p \psi(x_p + q \epsilon) \right).$$
  \label{theorem:KST-Sprecher}
\end{theorem}

Previous scholars, notably K\"oppen in \cite{koppen} and Braun and Griebel in \cite{braun-griebel}, are able to construct H\"older continuous inner functions, but no research has shown how to compute a Lipschitz continuous inner function. Before proceeding to discuss constructions of satisfactory Lipschitz functions, we first outline Kolmogorov's original proof of \ref{theorem:KST}, which will be pertinent to our later analysis.

\section{Kolmogorov's Original Proof}

Kolmogorov begins by dividing the line into intervals separated by gaps, which he denotes by $A_i$ with $i$ numbering
the intervals. He then replicates this division $2n + 1$ times where $n$ is the dimension, but shifts it so that the
gaps do not line up. Indexing the replicates by $q$, he now has $A^q_i$. To decompose the unit hypercube, he
takes all possible products of intervals, making small cubes $\left\{ S^q_{i_1,\dots,i_n} = \prod^n_{p=1}
A^q_{i_p}\,\,\vert\,\, 1 \le i_p \le m,\, 0 \le q \le 2n \right\}$. In two dimensions, for each $q$ we get what looks
like a system of city blocks separated by roads, which led Arnold to term them ``towns''~\cite{arnold,tikhomirov}. As a
last step, he makes a series of refinements to the line division, indexed by $j$, so that we have $A^q_{j,i}$ and carry
out the same construction for each level of refinement. We will henceforth refer to $j$ as the \textit{level} of
refinement. The idea is then to approximate part of the function on each shift so that they add up the right value, with
the gaps allowing us to keep the functions continuous.

A general proof of theorem \ref{theorem:KST} requires the following lemmas to define the inner functions of a KST representation, first stated in \cite{kolmogorov}:

  \begin{lemma} For each $q \in \{0,\dots,2n\}$ and at each refinement level $j \in \N$, there exists a system of cubes
    $$ \curly{S}^q_j = \left\{ S_{j;i_1,\dots,i_n}^q = \prod_{p=1}^n A^q_{j,i_p}\,\,\vert\,\, 1 \le i_p \le m_j,\, 0 \le q \le 2n \right\}$$
     that nearly cover the unit cube $I^n$, such that for any $x \in I^n$, there are $n+1$ values for $q$ such that $\curly{S}^q_j$ includes $x$.
    Additionally, $\forall q \in \{0,\dots,2n\}, \forall S \in \curly{S}^q_j,\,\, \text{Diam}[S] \rightarrow 0$ as $j \rightarrow \infty$.
    \label{lemma:cubes}
  \end{lemma}

\begin{lemma} There exist functions $\psi^{pq}$ such that for each $q$ and any $j \in \N$, the function $\Psi^q: \I^n \rightarrow \I$ defined as $\Psi^q(x_1,\dots,x_n) = \sum_{p=1}^n \psi^{pq}(x_p)$  satisfies the property that for any $S_1, S_2 \in \curly{S}^q_j$,
$$\Psi^q(S_1) \cap \Psi^q(S_2) = \emptyset. $$ \label{lemma:minsep}
\end{lemma}


We do not prove these lemmas here, but note that the following lemma is sufficient to prove \ref{lemma:minsep}:
\begin{lemma}
The constants $\lambda^{pq}_{j,i}$ and $\epsilon_j$ can be chosen so that the following conditions hold:
\begin{enumerate}
  \item $\lambda^{pq}_{j,i} < \lambda^{pq}_{j,i+1} \le \lambda^{pq}_{j,i} + \frac{1}{2^k}$.
  \item $\lambda^{pq}_{j,i} \le \lambda^{pq}_{j+1,i'} \le \lambda^{pq}_{j,i} + \epsilon_j - \epsilon_{j+1}$ if the closed intervals $A^q_{j,i} \cap A^q_{j+1,i'} = \emptyset$.
  \item The closed intervals $\Delta^q_{j; i_1,\dots,i_n} = \left[ \sum_{p=1}^n \lambda^{pq}_{j,i_p}, \sum_{p=1}^n \lambda^{pq}_{j,i_p} + n \epsilon_k  \right]$ are pairwise disjoint for fixed $j$ and $q$.
\end{enumerate}
Then, for fixed $p,q$, the following condition uniquely determines a continuous function $\psi^{pq}$ on $[0,1]$:
$$\lambda^{pq}_{j,i} \le \psi^{pq}(x) \le \lambda^{pq}_{j,i+1} \text{ for } x \in A^q_{j,i}.$$
\label{lemma:lambdas}
\end{lemma}

The rest of the proof proceeds following \cite{kolmogorov}. As this paper focuses on constructing the inner function for KST representation, the proof is not completed here but can be found in the appendix.

\section{A Misleading Candidate for a Lipschitz Inner Function}

One's first thought would be that to enforce lemma \ref{lemma:minsep}, it suffices to construct a Lipschitz monotonic function that separates out values on each of the squares, choosing the values of our function on those squares so that they do not coincide. We will construct such a function to illuminate why this alone fails to satisfy the conditions necessary for theorem \ref{theorem:KST}.
\\

We define intervals following Kolmogorov's idea of uniform spacing with shrinking diameters, combined with Sprecher's idea of decimal representation. Fix $n$. Let $\gamma \ge 2n+2$ be our base for decimal expansion. Let $\curly{D}_k$ be the set of rational numbers whose rational expansions in base $\gamma$ terminate at or before the $k^{th}$ decimal place. The set $\curly{D} = \cup_{k \in \N} \curly{D}_k$ is dense in $\R$.

Choose $\epsilon \in \left(\frac{1}{\gamma^2},\frac{1}{\gamma}\right)$. Let $\alpha_1 = 1$ and $\alpha_p = 2^{\frac{p-1}{n}}$ for $2 \le p \le n$.
For each $k \in \N$, define for each $d_k \in \curly{D}_k$ a corresponding interval $$A^0[d_k] = \left[ d_k, d_k + \frac{\gamma^2 - 1}{\gamma^{k+2}}  \right].$$
Then, for all $q \in \{1,\dots,2n\}$, define $$ A^{q}[d_k] = \left\{ x + q \epsilon \,\,\vert\,\, x \in A^0[d_k] \right\}.$$
Let $$\curly{A}^q_k = \left\{ A^q[d_k] \cap [0,1] \,\,\vert\,\, d_k \in \curly{D}_k \right\}.$$
Each interval in $A^q[d_k]$ has length $\frac{\gamma^2 - 1}{\gamma^{k+2}}$, and for each $k$ and a fixed $q$, the system of corresponding intervals is a translation of the original by a distance $\frac{1}{\gamma^k}$, with a gap of length $\frac{1}{\gamma^{k+2}}$ between each interval. For each $k \in \N$, for any $x \in [0,1]$, there are $2n$ values of $q$ (out of $2n+1$) such that $x \in \curly{A}^q_k$.
This type of construction is demonstrated in Fig.~\ref{fig:uniform}.

For each $q = 0,\dots,2n$ define $$S^q[d^{(1)}_k,\dots,d^{(n)}_k] = \prod_{p=1}^n A^q[d^{(p)}_k],$$ where $d^{(p)}_k \in \curly{D}_k$ for each $p = 1,\dots,n$ and the product denotes Cartesian product. Let $$\curly{S}^q_k = \left\{ S^q[d^{(1)}_k,\dots,d^{(n)}_k] \cap \I^n \,\,\vert\,\, d^{(p)}_k \in \curly{D}_k,\, p = 1,\dots n\right\}.$$
The following lemma is easy to verify.
\begin{lemma} For any $k \in \N$ and for each $x \in \I^n$, there are $n+1$ values of $q$ such that $x \in \curly{S}^q_k$. \label{lemma:cubesBadFxn}
\end{lemma}
\begin{center}
\begin{figure}[tb]
  \resizebox{\textwidth}{!}{
\begin{tikzpicture}
\draw [help lines] (0, 0) -- (10, 0) ;
\draw [ultra thick] (0,0) -- (0.9,0) ;
\draw [ultra thick] (1,0) -- (1.9,0) ;
\draw [ultra thick] (2,0) -- (2.9,0) ;
\draw [ultra thick] (3,0) -- (3.9,0) ;
\draw [ultra thick] (4,0) -- (4.9,0) ;
\draw [ultra thick] (5,0) -- (5.9,0) ;
\draw [ultra thick] (6,0) -- (6.9,0) ;
\draw [ultra thick] (7,0) -- (7.9,0) ;
\draw [ultra thick] (8,0) -- (8.9,0) ;
\draw [ultra thick] (9,0) -- (9.9,0) ;
\draw (0,0) -- (0.9,0.0) -- (1,0.5) -- (1.9,0.5) -- (2,1) -- (2.9,1) -- (3,1.5) -- (3.9,1.5) -- (4,2) -- (4.9,2) -- (5,2.5) -- (5.9,2.5) -- (6,3) -- (6.9,3) -- (7,3.5) -- (7.9,3.5) -- (8,4) -- (8.9,4) -- (9,4.5) -- (9.9,4.5) -- (10,5) ;
\draw [dashed] (0,0) -- (10,5) ;
\node [below] at (5,-0.1) {$\curly{A}^q_k$} ;
\node [below right] at (10,5) {$\psi^{p,q}_k$} ;
\node [left]  at (9.9, 5) {$\psi^{p,q}$} ;
\node [below] at (0,0) {0} ;
\node [below] at (10,0) {1} ;
\end{tikzpicture}
}
\caption{Construction of intervals $A^q[d_k]$ for $n=2$ and $\gamma = 10$.}
\label{fig:uniform}
\end{figure}
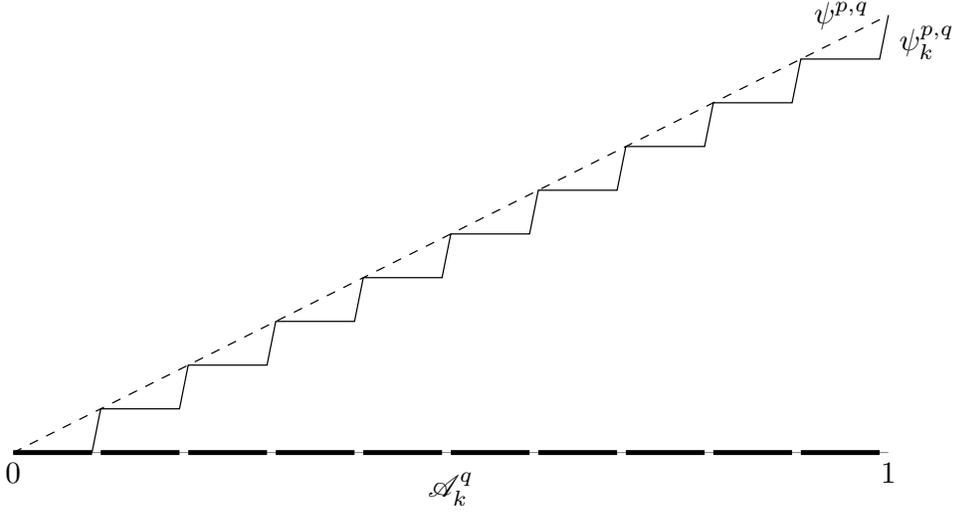
\end{center}
For each $q = 0,\dots,2n$, at each level of refinement $k \in \N$ define $\psi^{p,q}_k : \I \rightarrow \R$ by setting $\psi^{p,q}_k(A^q[d_k]) = \alpha_p (d_k + q \epsilon)$ for each $A^q[d_k] \in \curly{A}^q_k$, and interpolating linearly on the gaps between successive intervals. In Fig.~\ref{fig:uniform}, we see such a function drawn for the set of intervals shown.
Similarly define $\psi^{p,q}: \I \rightarrow \R$ as $$\psi^{p,q}(x) = \alpha_p (x + q \epsilon);$$
note that $\lim_{k \rightarrow \infty} \psi^{p,q}_k = \psi^{p,q}$ uniformly.
Define for each $q \in \{0,\dots,2n\}$ the function $\Psi: \I^n \rightarrow \R$ as $$\Psi^q(x_1,\dots,x_n) = \sum_{p=1}^n \psi^{p,q}(x_p).$$

Define $\lambda^{p,q}_{k,i} = \psi^{p,q}\left(\frac{i}{\gamma^k} \right)$ for $i = 0,\dots, \gamma^k$. Denote $\widehat{\alpha} = \max_{i=1,\dots,p} \alpha_p$ and define $\epsilon_k = \frac{\widehat{\alpha}}{\gamma^k}.$
  \begin{lemma}
    Choose the constants $\lambda_{k,i}^{p,q}$ and $\epsilon_k$ so that the following conditions hold:
    \begin{enumerate}
    \item $\lambda_{k,i}^{p,q} < \lambda_{k,i+1}^{p,q} \le \lambda_{k,i}^{p,q} + O\left(\frac{1}{2^k}\right)$.
    \item $\lambda_{k,i}^{p,q} \le \lambda_{k+1,i'}^{p,q} \le \lambda_{k,i}^{p,q} + \epsilon_k - \epsilon_{k+1}$ if the closed intervals $A^q_{k,i}$ and $A^q_{k+1,i'}$
    do not intersect and $A^q_{k+1,i'}$ falls into the gap between $A^q_{k,i}$ and $A^q_{k,i+1}$.
    \item For fixed $k, q$,  for any $(i_1,\dots,i_n) \ne (j_1,\dots,j_n) \in \{0,\dots,\gamma^k\}^n$, we have
    $$\Psi^q\left( \frac{i_1}{\gamma^k},\dots,\frac{i_n}{\gamma^k}\right) \ne \Psi^q\left( \frac{j_1}{\gamma^k},\dots,\frac{j_n}{\gamma^k}\right).$$
\end{enumerate}
  \label{lemma:lambdasBadFxn}
  \end{lemma}
\begin{proof}
\begin{enumerate}
  \item For each $k \in \N$ and $i = 0,\dots,\gamma^k$, we have
  \begin{equation*} \begin{split}
\lambda^{p,q}_{k,i} &= \alpha_p \left( \frac{i}{\gamma^k} + q \epsilon \right) \\
\lambda^{p,q}_{k,i+1}&= \alpha_p \left( \frac{i+1}{\gamma^k} + q \epsilon \right)\\
\lambda^{p,q}_{k,i+1} - \lambda^{p,q}_{k,i} &= \frac{\alpha_p}{\gamma^k} \in O(2^{-k}) \text{ since } \gamma > 2.
  \end{split} \end{equation*}
  \item For each $k \in \N$, $i = 0,\dots,\gamma^k$, and $j=1,\dots,\gamma-1$,
  \begin{equation*} \begin{split}
\lambda^{p,q}_{k,i} &= \alpha_p \left( \frac{i}{\gamma^k} + q \epsilon \right) \\
\lambda^{p,q}_{k+1, \gamma i + j} &= \alpha_p \left( \frac{\gamma i + j}{\gamma^{k+1}} + q \epsilon \right) \\
\lambda^{p,q}_{k+1, \gamma i + j} - \lambda^{p,q}_{k,i} &= \alpha_p \left( \frac{j}{\gamma^{k+1}}\right) \\
 &\le \widehat{\alpha} \left( \frac{j}{\gamma^{k+1}} \right) \\
&\le \widehat{\alpha}\left( \frac{\gamma - 1}{\gamma^{k+1}} \right) \\
&= \widehat{\alpha} \left( \frac{1}{\gamma^k} - \frac{1}{\gamma^{k+1}}\right) \\
&= \epsilon_{k} - \epsilon_{k+1}.
  \end{split} \end{equation*}
  \item Note that the numbers $\alpha_1,\dots,\alpha_n$ are rationally independent, i.e. $\forall x = (x_1,\dots,x_n)\in \Q^n$, if $x \ne 0$, then $$\sum_{p=1}^n \alpha_p x_p \ne 0.$$
Fix $k$ and $q$; choose any two $(i_1,\dots,i_n) \ne (j_1,\dots,j_n) \in \{0,\dots,\gamma^k\}^n$. Then,
\begin{equation*} \begin{split}
\Psi^q\left( \frac{i_1}{\gamma^k},\dots,\frac{i_n}{\gamma^k}\right) - \Psi^q\left( \frac{j_1}{\gamma^k},\dots,\frac{j_n}{\gamma^k}\right) &=
 \sum_{p=1}^n \psi^{p,q}\left( \frac{i_p}{\gamma^k} \right) - \sum_{p=1}^n \psi^{p,q}\left( \frac{j_p}{\gamma^k}\right) \\
 &= \sum_{p=1}^n \alpha_p \left( \frac{i_p}{\gamma^k} - \frac{j_p}{\gamma^k} \right) \\
 &\ne 0,
\end{split} \end{equation*} since $\frac{i_p - j_p}{\gamma^k}$ is rational for all $p = 1,\dots,n$ and not zero.
\end{enumerate}
\end{proof}

\subsection{Why this function is misleading}

We see that lemma \ref{lemma:cubesBadFxn} is sufficient for lemma \ref{lemma:cubes}. Lemma \ref{lemma:lambdasBadFxn} is nearly identical as lemma \ref{lemma:lambdas}. We would expect then that our functions $\psi^{pq}(x) = \alpha_p(x+q\epsilon)$ are valid inner functions for theorem \ref{theorem:KST}. \\

However, this is not the case. We will illustrate this for $n=2,$ but the flaw extends to higher dimensions. Let $\gamma = 10$. Choose $\alpha_1 = 1,\,\alpha_2 = \sqrt{2}.$ Fix $q \in \{0,\dots,4\}$; we arbitrarily choose $q = 0$. Let $$\Psi(x) = \Psi^0(x) = \sum_{p=1}^2 \psi^{p0}(x_p + 0 \epsilon) = x_1+\sqrt{2}x_2.$$
The points $x^{(1)} = \left(0, \frac{1}{2\sqrt{2}}\right)$ and $x^{(2)} = \left(\frac{1}{2}, 0\right)$ fall in different boxes $S^0_{(1)},\,S^0_{(2)} \in \curly{S}_1^q$, but $$\Psi(x^{(1)}) = \Psi(x^{(2)}) = \frac{1}{2}.$$

For every level of refinement $j \in \N$, the function $\Psi_j^q = \sum_{p=1}^n \psi_j^{pq}$ satisfies the separation lemma \ref{lemma:minsep}, but in the limit, we lose separation: in essence, we only separate function values defined on points in our dyadic expansion $\curly{D}_j$.

More pressing is the result by Vitushkin in \cite{vitushkin}, that such an inner function cannot be continuously differentiable if we aim to represent smooth multivariate functions as superpositions of univariate functions. Vitushkin and Henkin give a stronger result in \cite{vitushkin-henkin}, highlighted by Lorentz in \cite{lorentz}.
\begin{theorem}
Let $D \subset \R^n$ be a compact connected domain with non-empty interior. Fix $m \in \N$ and functions $p_i,\,q_i \in C(D)$, where $i \in \{1,\dots,m\}$, with for each $i$, the function $q_i$ is continuously differentiable. Let $$F = \left\{ \sum_{i=1}^m \, p_i \,(g_i \circ q_i) \,\,:\,\, g_i \in C(\R) \right\}.$$ Then, $F$ is nowhere dense in $C(D),$ and is a set of first category in $C(D)$.
In particular, there is even a polynomial that is not contained in $F$.
\end{theorem}
 Even though at each step, the functions $\psi^{pq}_j$ are not continuously differentiable, their limit $\psi^{pq}$ is continuously differentiable. Setting $m = 2n+1$, $p_i \equiv 1$ for $i \in {1,\dots,m}$, and choosing the functions $g_i = \Psi^i = \sum_{p=1}^n \psi^{pi}$ our constructed inner functions confirms that a linear inner function does not allow for Kolmogorov Representation.

\section{Construction of our Lipschitz Inner Function}

We now turn to the construction of our own Lipschitz continuous inner function that meet the conditions for \ref{theorem:KST-Sprecher}. We construct functions $\Psi^q(x_1,\dots,x_n) = \sum_{p=1}^n \lambda^p \psi(x_p + q \epsilon)$ that embed $[0,1]^n$ into $\R^{2n+1}$. We borrow notation from Sprecher \cite{sprecher1972}; as such, we restate lemmas \ref{lemma:cubes} and \ref{lemma:minsep}.

\begin{lemma}\label{lemma:cubesLipschitz}
Fix $j \in \N$ and $\epsilon \in (0,1/2n]$, and let $\curly{T}_j$ be a set of closed intervals in $[-1,1]$. Define $\curly{T}^q_j$ as a set of closed intervals in $[-1+q \epsilon, 1 + q \epsilon]$ for $q \in \{0.\dots,2n\}$ such that $$\curly{T}^q_j = \{t + q \epsilon \,\,:\,\,t \in \curly{T}_j\}.$$
 Enumerate the intervals in $\curly{T}^q_j$ so that $$\curly{T}^q_j = \{ t^q_i \,\,:\,\, i=1,\dots,m_j \} \text{ with }m_j \in \N.$$
 Define $$S^q_{j; i_1,\dots,i_n} = \prod_{p=1}^n t^q_{i_p} \text{ where }t^q_{i_p} \in \curly{T}^q_j \text{ for }p = 1,\dots,n,$$ and let
 $$\curly{S}^q_j = \{ S^q_{j; i_1,\dots,i_n} \,\,:\,\, 1 \le i_1,\dots,i_n \le m_j \}.$$  Suppose the families of cubes $\curly{S}^q_j$ satisfy the following:
\begin{enumerate}
\item {\normalfont Boxes in a town are disjoint on each level}\\
  For any $q$, if $(i_1,\dots,i_n) \ne (i'_1,\dots,i'_n)$, then $$S^q_{j; i_1,\dots,i_n} \cap S^q_{j,i'_1,\dots,i'_n} = \emptyset.$$
\item {\normalfont Each point intersects $n+1$ towns on a level}\\
  $\forall x \in [0,1]^n,\,\exists q_1,\dots,q_{n+1}$ such that $x \in S^{q_k}_{j; i^k_1,\dots,i^k_n}$ for some valid indices $(i^k_1,\dots,i^k_n)$ for each $k \in 1,\dots,n+1$.
\item {\normalfont Boxes get smaller uniformly with increasing level}\\
  $\text{Diam}[S^q_{j; i_1,\dots,i_n}] \rightarrow 0$ uniformly as $j \rightarrow \infty$ in $(i_1,\dots,i_n)$ for each valid set of indices and for every $q$.
\end{enumerate}
Let $\Psi^q(x_1,\dots,x_n) = \sum_{p=1}^n \lambda_p \psi(x_p + q \epsilon)$, where $\lambda_1,\dots,\lambda_n$ are integrally independent. Suppose that each of the families of cubes $\curly{S}^q_j$ additionally satisfies:
\begin{enumerate}
  \setcounter{enumi}{3}
  \item $\Psi^q( S^q_{j; i_1,\dots,i_n}) \cap \Psi^q(S^q_{j,i'_1,\dots,i'_n}) = \emptyset$ when $(i_1,\dots,i_n) \ne (i'_1,\dots,i'_n)$.
\end{enumerate}
Then, any function $f \in C([0,1]^n)$ admits a KST representation $$f = \sum_{q=0}^{2n} \chi^q \circ \Psi^q .$$
\end{lemma}

Following the proofs in \cite{kolmogorov}, \cite{fridman}, and \cite{sprecher1972}, we construct $2n+1$ sets of intervals $\curly{T}^q \subseteq [-1,2],\, q \in \{0,\dots,2n\},$ such that every point in $[0,1]$ is contained in at least $2n$
of the $2n+1$ sets of intervals, or towns. Define $\curly{T}^q$ iteratively as $\curly{T}^q_j$ as $j \rightarrow \infty$, where $\curly{T}^q_j$ is a translation of $\curly{T}_j$ by $q \epsilon$, where $\epsilon \in \left( 0, \frac{1}{2n}\right]$:
\begin{equation*} \begin{split}
\curly{T}^q &= \left\{ t + q \epsilon \,\,:\,\, t \in \curly{T} \right\} \\
\curly{T}^q_j &= \left\{ t + q \epsilon \,\,:\,\, t \in \curly{T}_j \right\}. \\
\end{split} \end{equation*}
It is sufficient to construct a system of towns that satisfies the following:
\begin{lemma}  \label{lemma:criterion} For each $j \in \N$, the system of towns $\curly{T}_j$ and the function $\psi_j: [-1,1]\rightarrow \R \in C[-1,1]$ satisfy the following:
\begin{enumerate}
  \item {\normalfont Intervals get smaller uniformly with increasing level}\\
  $\sup_{t \in \curly{T}_j}\text{Diam}(t) \rightarrow 0$ uniformly as $j \rightarrow 0$.
  \item {\normalfont Each point intersects $2n$ towns on a level}\\
  For each point $x \in [0,1]$, $\exists q_1,\dots,q_{2n} \in \{0,\dots,2n\}$ such that there is some $t^{q_k} \in \curly{T}^{q_k}_j$ such that $x \in t^{q_k}$, where $k =1,\dots,2n$.
  \item {\normalfont Images of disjoint intervals are disjoint}\\
  $\forall t_1 \ne t_2 \in \curly{T}_j$, $\psi(t_1) \cap \psi(t_2) = \emptyset.$
  \item {\normalfont Functions $\psi_j$ are Lipschitz}\\
  The maximum slope of $\psi_j$ is $\widehat{m}_j = 1 - 2^{-j}$.
\end{enumerate}
\end{lemma}

Our algorithm proceeds by examining the intervals at a given level, and breaking larger ones in order to enforce the
diameter condition of lemma~\ref{lemma:cubesLipschitz}. The algorithm is robust to a permutation of the order in which it
processes towns. Given a system of towns, we first determine which intervals should be broken. We break these intervals by
removing a \textit{gap} that includes the interval midpoint. Removing these gaps might cause some of these midpoints to
no longer be included in at least $2n$ intervals, so if a midpoint falls in a hole between two intervals in another
town, we insert a \textit{plug} into that hole. We determine the width of these plugs by solving a block-diagonal linear
system. By adding in these plugs, we make sure that every break point is contained in $2n+1$ towns, so that when we
break apart these intervals, each break point is still contained in at least $2n$ towns, satisfying
lemma~\ref{lemma:criterion}. After adding in all necessary plugs, we proceed to break apart each of the intervals that
was above our threshold. If the gap we create is small enough, breaking apart one interval has no effect on other
intervals at the same refinement level.

In addition to satisfying lemmas~\ref{lemma:cubesLipschitz} and~\ref{lemma:criterion}, we desire that the function $\psi$ is
robust to choices made during its computation, such as which towns to process first at a given refinement level. If we
proceed town-by-town, we may create plugs or gaps that shift into other plugs or gaps at the same level of refinement,
changing the function values assigned to other towns at the same level. However, since we solve for all plugs at once on
a given level, our functions do have this robustness.

Algorithm \ref{alg:Lip} outlines our implementation. We begin with $\psi_0 \equiv 0$ and $\curly{T}_0 = \{ [-1,1] \}.$ For each refinement level $j \in \N$, there are three primary stages:
\begin{enumerate}
  \item Find Holes
  \item Solve for Plugs
  \item Create Gaps
\end{enumerate}
We now describe each stage of this implementation in greater detail.

\begin{algorithm}
\caption{Lipschitz Inner Function}\label{alg:Lip}
\begin{algorithmic}
\Procedure{Lip}{$n$}
  \Comment{$n$ spatial dimension} \\
  \While{$j < \infty$}
      \State Get all towns above threshold, to break during iteration $j$:
      \State $\widehat{\curly{T}} \gets \{ I \in \curly{T} \,\,:\,\, \lvert I \rvert \ge \left(1/2\right)^j \}$ \\

      \State Find Holes:
      \State $\curly{H} \gets \emptyset$
      \For{$ t \in \widehat{\curly{T}} $}
          \State Get point to break, and check the number of gaps it falls in:
          \State $p \gets t.center$, slightly perturbed (if needed)
          \State $k \gets $ numbers of holes $p$ falls in
          \If{$k > 0$}
              \ForAll{ hole $h$ that include $p$}
                  \State $\curly{H} \gets \curly{H} \cup \{h\}$ \\
              \EndFor
          \EndIf
      \EndFor

      \State Solve for Plugs:
      \State $\curly{P} \gets \emptyset$
      \ForAll{$h \in \curly{H}$}
          \State Solve Linear System for plug widths for all the plugs in hole $h$
          \State Construct Town $\pi_i$ from each plug
          \State Add plugs $\pi_i$ to $\curly{P}$
      \EndFor
      \State $\curly{T} \gets \curly{T} \cup \curly{P}$ \\

      \State Create Gaps:
      \For{$ t \in \widehat{\curly{T}}$}
          \State Find gap around center of $t$:
          \State $gap_s, gap_e \gets $ GetGap($t$) \\

          \State Get new function $\psi$ value:
          \State $nxt \gets $ next town after $t$
          \State new slope $\gets \max\{ \frac{1}{2},\,\frac{1}{2}($slope between towns $t$ and $nxt)\}$
          \State $v \gets t.val + ($new slope$)(gap_e - gap_s)$ \\

          \State Break $t$ into two new Towns $tl,\,tr$:
          \State $tl \gets $Town($t.start, gap_s, t.val)$
          \State $tr \gets $Town($gap_e, t.end, v)$
          \State $\curly{T} \gets \curly{T} \cup \{tl,\,tr\}$
      \EndFor

  \EndWhile
\EndProcedure
\end{algorithmic}
\end{algorithm}

\subsection{Finding Holes}

Finding which holes need plugs is straightforward. Take $\curly{T}_j$ defined as before at refinement level $j \in \N$.
For each $t \in T = \{t \in \curly{T}_j \,\,:\,\, \lvert t \rvert \ge \theta^j \}$ and for $p$ the breakpoint of $t$, we define
\begin{equation*} \begin{split}
\widehat{T}_j &= \{t \in \curly{T}_j \,\,:\,\, \exists q \in \{-2n,\dots,2n\} \text{ such that } p - q \epsilon \in t \} \\
\widehat{Q}_j &= \{ q \in \{-2n,\dots,2n\} \,\,:\,\, \exists t \in \curly{T}_j \text{ such that } p - q \epsilon \in t \} \\
\widehat{Q}^\emptyset_j &= \{ q \in \{-2n,\dots,2n\} \,\,:\,\, \forall t \in \curly{T}_j,\, p - q \epsilon \notin t \text{ and } p - q \epsilon \in [-1,1] \}.
\end{split} \end{equation*}
The break point $p$ is the midpoint of $t$, unless there is some other break point that is an integer multiple of $\epsilon$ away from $p$; in this case, perturb $p$ by some small rational amount. Denote the set of break points $$\widehat{P}_j = \{p \text{ break point of } t \,\,:\,\, t \in T\}.$$
If $\lvert \widehat{Q}_j\rvert < 2n+1$, then $p$ falls in a hole, denoted as $h_p$, defined by the open interval between the two closest intervals. Let $\curly{H}$ be the set of all holes to plug. In the case $\lvert \widehat{Q}^\emptyset_j \rvert > 1$, we add multiple holes to $\curly{H}$, and that a hole in $\curly{H}$ might contain multiple plugs, if there are multiple break points who, once shifted, fall into that hole.

\subsection{Solve for Plugs}

For each hole $h  \in \curly{H}$, we proceed as follows. Let $\nu \ge 1 \in \N$ be the number of points who, once
shifted, fall into hole $h$, denoted by its endpoints $h = (b_0, a_{\nu+1})$.
Note that $\psi_j$ is linear on $h$; suppose its slope is $m$. We wish to construct $\nu$ ``plugs", i.e. closed intervals denoted $\pi_i = [a_i, b_i] \subset h$, for $1 \le i \le \nu$, such that for each $p_i$ with appropriate shift index $q_i$,
$$p_i - q_i \epsilon \in \pi_i,$$
and for $i_1 \ne i_2 \in \{1,\dots,\nu\}$, $$\pi_{i_1} \cap \pi_{i_2} = \emptyset.$$
For simplicity, we define $\widehat{p}_i = p_i - q_i \epsilon.$
The plugs $\pi_i$ are constrained so that $\psi_{j+1}(\pi_i) = \psi_j(\widehat{p}_i)$ and between plugs, $\psi_{j+1}$ has slope $\widehat{m} = 1 - 2^{-j-1}$.
Denote
\begin{equation*} \begin{split}
f_0 &= \psi_j(b_0) \\
f_i &= \psi_j(\widehat{p}_i) \qquad 1 \le i \le \nu \\
f_{\nu+1} &= \psi_j(a_{\nu+1}). \\
\end{split} \end{equation*}
The slope constraints provide $\nu+1$ equations
$$ \widehat{m} (a_i - b_{i-1}) = f_i - f_{i-1} \qquad 1 \le i \le \nu + 1.$$
Since there are $2\nu$ variables but $\nu+1$ constraints, between each plug we enforce a symmetry constraint\footnote{We could have employed other constraints.}
 on $\psi_{j+1}$, that
$$b_i - \widehat{p}_i = \widehat{p}_{i+1} - a_{i+1} \qquad 1 \le i \le \nu-1.$$
Fig.~\ref{fig:plugs} illustrates this setup in the case $\nu = 2$.
\begin{center}
\begin{figure}[tb]
\begin{tikzpicture}
\draw [blue, thick] (3,1) -- (4, 2) -- (4.5833, 2) -- (6.0833, 3.5) -- (6.5, 3.5) -- (7,4) ;
\draw [help lines] (3,0) -- (3,1) ;
\draw [help lines] (4,0) -- (4,2) ;
\draw [help lines] (4.3333, 0) -- (4.3333, 2) ;
\draw [help lines] (4.5833, 0) -- (4.5833, 2) ;
\draw [help lines] (6.0833, 0) -- (6.0833, 3.5) ;
\draw [help lines] (6.3333, 0) -- (6.3333, 3.5) ;
\draw [help lines] (6.5, 0) -- (6.5, 3.5) ;
\draw [help lines] (7, 0) -- (7, 4) ;
\draw [thick] (0,1) -- (3,1) -- (7,4) -- (10,4) ;
\draw [help lines] (0,0) -- (10,0) ;
\draw [line width = 6] (0,0) -- (3,0) ;
\draw [line width = 6] (7,0) -- (10,0) ;
\draw [blue, line width = 6] (4,0) -- (4.5833,0) ;
\draw [blue, line width = 6] (6.0833,0) -- (6.5,0) ;
\draw [help lines] (0,2) -- (3.9,2) ;
\draw [help lines] (0,3.5) -- (5.9833,3.5) ;
\draw [help lines] (0,4) -- (6.9,4) ;
\node [below] at (3,-0.1) {$b_0$} ;
\node [below] at (4,-0.1) {$a_1$} ;
\node [below] at (4.3333,-0.5) {$\widehat{p}_1$} ;
\node [below] at (4.5833,-0.1) {$b_1$} ;
\node [below] at (6.0833,-0.1) {$a_2$} ;
\node [below] at (6.3333,-0.5) {$\widehat{p}_2$} ;
\node [below] at (6.5,-0.1) {$b_2$} ;
\node [below] at (7,-0.1) {$a_3$} ;
\node [left] at (-0.5,1) {$f_0$} ;
\node [left] at (-0.5,2) {$f_1$} ;
\node [left] at (-0.5,3.5) {$f_2$} ;
\node [left] at (-0.5,4) {$f_3$} ;
\end{tikzpicture}
\caption{Sketch of scenario for finding two plugs, with $\psi_j$ in black and $\psi_{j+1}$ in blue. Note the symmetry constraint $b_1 - \widehat{p}_1 = \widehat{p}_2 - a_2$ is enforced.}
\label{fig:plugs}
\end{figure}
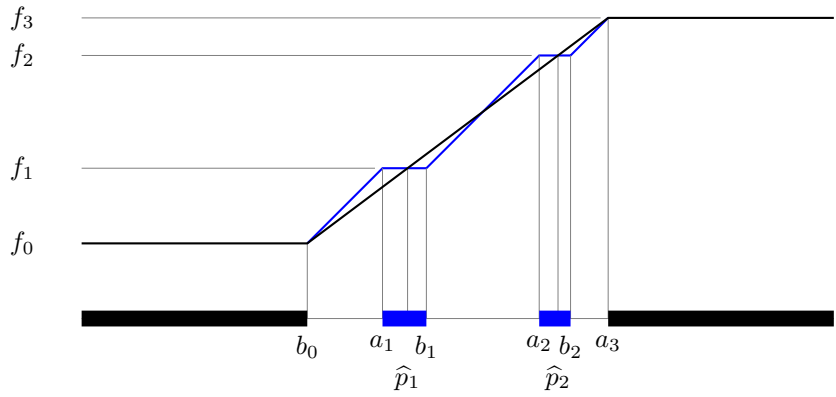
\end{center}

This provides the linear equation $Cx = z,$ where
\begin{equation*} \begin{split}
  \begin{matrix*}[r] 
      \begin{matrix*}[r]    
        \nu+1 \text{ slope equations }
          \left\{\vphantom{\begin{matrix}  \\ \\\\ \\ \\ \\ \end{matrix}}\right. \\
        \nu-1 \text{ symmetry equations }
          \left\{\vphantom{\begin{matrix}  \\ \\ \\ \\ \end{matrix}}\right. \\
      \end{matrix*}
    &
  \left( \begin{array}{ccccccccc}
      \widehat{m} &               &             &          &                &             &              \\
                  & -\widehat{m}  & \widehat{m} &          &                &             &              \\
                  &               &             &  \ddots  &                &             &              \\
                  &               &             &          & -\widehat{m}   & \widehat{m} &              \\
                  &               &             &          &                &             & -\widehat{m} \\
                  \hline \\[-\normalbaselineskip]
      0 & 1 & 1 & 0         &   &   &   \\
        &   &   &   \ddots  &   &   &   \\
        &   &   &         0 & 1 & 1 & 0 \\
      \end{array} \right)
  \end{matrix*} &= C, \\
  \begin{matrix*}[r] 
      \begin{matrix*}[r]    
        \nu+1 \text{ slope equations }
          \left\{\vphantom{\begin{matrix}  \\ \\\\ \\ \\ \\ \end{matrix}}\right. \\
        \nu-1 \text{ symmetry equations }
          \left\{\vphantom{\begin{matrix}  \\ \\ \\ \\ \end{matrix}}\right. \\
      \end{matrix*}
    & \qquad \qquad
      \left(\begin{array}{c}
      f_1 - f_0 \\
      \phantom{a} \\
      \vdots \\
      \phantom{a} \\
      f_{\nu+1} - f_\nu \\
            \hline \\[-\normalbaselineskip]
      \widehat{p}_1 + \widehat{p}_2 \\
      \vdots \\
      \widehat{p}_{\nu-1} + \widehat{p}_\nu
      \end{array}\right)  +
      \begin{pmatrix}
      \widehat{m} b_0 \\ \phantom{\vdots}\\ \\ \\ -\widehat{m} a_{\nu+1} \\ \\ \phantom{\vdots} \\ \\
      \end{pmatrix}
  \end{matrix*}  &= z, \\
  \begin{pmatrix}
  a_1 \\ b_1 \\ \vdots \\ a_i \\ b_i \\ \vdots \\ a_n \\ b_n
  \end{pmatrix} &=x.
\end{split} \end{equation*}
Permuting the rows of $C$ creates a block diagonal matrix; since each block is invertible, $C$ is invertible, so a unique solution exists. Given $\psi_{j}$ is monotonic increasing, it is easy to show that the plugs $\pi_i$ are well-defined and do not overlap. On each plug, assign function values
$$\psi_{j+1}(\pi_i) = \psi_j(p_i).$$
For each $h$, add the plugs $\pi_i$ to $\curly{T}_j$.

\subsection{Create Gaps}

Fix $\alpha = 2/3$ and $\beta =1/3$.
For each $p \in \widehat{P}_j$, we proceed as follows. We know $\forall q \in \{-2n,\dots, 2n\}$, if $p - q \epsilon \in [-1,1]$, then $$\exists t^q \in \curly{T}^q_j \text{ such that } p \in t^q.$$

Let $t = t^0 = [a,b]$, and define
\begin{equation*} \begin{split}
\rho_+ &= \min_{ t^q = [a^q, b^q] \in \widehat{T}_j }\{ (p - q \epsilon) - a^q \,\,:\,\,  p-q\epsilon \in t^q \} \\
\rho_- &= \min_{ t^q = [a^q, b^q] \in \widehat{T}_j }\{ b^q - (p - q \epsilon) \,\,:\,\,  p-q\epsilon \in t^q \} \\
\end{split} \end{equation*}

To avoid creating gaps that overlap, we also define
\begin{equation*} \begin{split}
  \delta_+ &= \min_{ \substack{ \tilde{p} \in \widehat{P}_j \\ q \in \{-2n,\dots,2n\} \\ \tilde{q} \in \{-2n,\dots,2n\} } }\{ (p - q \epsilon) - (\tilde{p} - \tilde{q}\epsilon) \,\,:\,\,  p - q \epsilon > \tilde{p} - \tilde{q}\epsilon  \} \\
  \delta_- &= \min_{ \substack{ \tilde{p} \in \widehat{P}_j \\ q \in \{-2n,\dots,2n\} \\ \tilde{q} \in \{-2n,\dots,2n\} } }\{ (\tilde{p} - \tilde{q}\epsilon) - (p - q \epsilon) \,\,:\,\,  p - q \epsilon < \tilde{p} - \tilde{q}\epsilon \} \\
\end{split} \end{equation*}

Set $$\rho = \min\{ \alpha \rho_+,\, \alpha \rho_-,\, \beta \delta_+,\, \beta \delta_-\}.$$
Take $g = ( p - \rho,\, p + \rho)$. This guarantees $(g - q \epsilon \cap [-1,1]) \subset t^q \,\,\forall q \in \{0,\dots,2n\}$.
Break $t$ into two new intervals, $t_- = [a, p - \rho]$ and $t_+ = [p + \rho, b]$. Then,
$$t = t_- \cup g \cup t_+, \qquad p \in g.$$

Let $t_n$ be the next interval greater than $t$. Assign function values
\begin{equation*} \begin{split}
 \psi_{j+1}(t_-) &= \psi_j(t) \\
 \psi_{j+1}(t_+) &= \psi_j(t) + \eta \\
\end{split} \end{equation*}
where $$\eta = \min\left\{ \rho,\, \frac{1}{2}(\psi_j(t_n) - \psi_j(t))\right\}.$$
This creates $\curly{T}_{j+1}$ from $\curly{T}_j$ by replacing $t$ with $t_-$ and $t_+$ for each $p \in \widehat{P}_j$.

\subsection{Analysis of Inner Function $\psi$}

\begin{claim} At each $j \in \N$, the system of towns $\curly{T}_j$ satisfies lemma \ref{lemma:criterion}
\end{claim}
\begin{proof}
  By construction, $\sup_{t \in \curly{T}_j} Diam(t) \le c\, \theta^j,$ where $\theta \in (0,1)$ and $c \in \R > 0$ independent of $j$. We will show by induction that $\forall j \in \N$, we have $\forall x \in [0,1]$, $x$ is contained in at least $2n$ towns $\curly{T}^q_j$.
  \begin{itemize}
\item Base Case: $\curly{T}^q_0 = \{[-1 + q \epsilon, 1 + q \epsilon]\}.$ Fix $x \in [0,1]$. Since $\epsilon \in (0,\frac{1}{2n}]$, we have  $\forall q \in \{0,\dots,2n\}$, $x \in [-1+q\epsilon, 1 + q \epsilon]$.
\item Inductive Step: Suppose $\forall x \in [0,1]$, this claim holds true through refinement level $j \in \N$. Then, it still holds true after adding in the plugs at refinement level $j$, since we have added more intervals to $\curly{T}_j$ whilst not removing any gaps. Fix $p \in \widehat{P}_j$, and let $g$ be the gap containing point $p$ that we create at this refinement level. By construction, $\forall q \in \{0,\dots,2n\}$, we have $g \subset t^q$ for some $t^q \in \curly{T}^q_j$.
By construction, $\forall p_1, p_2 \in \widehat{P}_j$ with gaps $g_1$ and $g_2$, we have $\bar{g}_1 \cap \bar{g}_2 = \emptyset$. Therefore, any point $x \in [0,1]$ that is in a gap we remove is still covered by the families $\curly{T}^q_{j+1}$, where $q \in \{1,\dots,2n\}$, and any other point is still covered by some interval in at least $2n$ of the towns, so the inductive step holds.
  \end{itemize}
\end{proof}

\begin{claim} The function $\psi = \lim_{j \rightarrow \infty} \psi_j$ is Lipschitz continuous on $[-1,1]$ with constant 1.
\end{claim}

\begin{proof}
Clearly, $\psi_j$ is continuous for $j \in \N$. We will show the following two lemmas.

\begin{lemma} The function $\psi = \lim_{j \rightarrow \infty} \psi_j$ is well-defined, with convergence in the sup norm.
\end{lemma}

\begin{proof}
Let $G_j = \{ g_p \,\,:\,\, g_p \text{ gap containing } p \in \widehat{P}_j \}$ be the set of all gaps and $\Pi_j =
\{\pi \,\,:\,\, \pi \text{ plug added at step } j \}$ be the set of all plugs. We have
\begin{equation*} \begin{split}
\norm{\psi_{j+1} - \psi_j}_\infty &\le \max\left\{ \sup_{g \in G_{j+1}} \widehat{m}_j \lvert g_p \rvert , \sup_{\pi \in \Pi_{j+1}} \lvert \pi \rvert \right\}\\
&< \sup_{g \in G_{j+1}} \lvert g_p \rvert \\
&< \sup_{t \in \curly{T}_j} \mathrm{Diam}(t) \\
&\le c\, \theta^j,
\end{split} \end{equation*}
where $\theta \in (0,1)$ and $c \in \R > 0$ is independent of $j$. Therefore, by the Weierstrass M-test, the function
$$\psi = \lim_{j \rightarrow \infty} \psi_j - \psi_0 =  \lim_{j \rightarrow \infty} \sum_{k=1}^j \psi_k - \psi_{k-1}$$ is well-defined.
\end{proof}

\begin{lemma} Assume that $\psi_j$ is monotonic increasing, constant on each interval $t \in \curly{T}_j$, and linear between such intervals with slope $\le 1 - 2^{-j}$.
Then, $\psi_{j+1}$ is also monotonic increasing, constant on each interval $t \in \curly{T}_{j+1}$, and linear between
such intervals, with slope $\le 1 - 2^{-j-1}$.
\end{lemma}

\begin{proof}
By construction, $\psi_{j+1}$ is constant on each interval. Between intervals, we interpolate linearly.
For each gap $g$ formed between intervals from $\curly{T}_{j+1}$, let $t$ be the interval to the left of $g$, and $t_n$ be to the right. Exactly one of the following three cases must be true for each $g$:
\begin{enumerate}
  \item The same gap $g$ existed between intervals $t$ and $t_n$ at refinement level $j$.
  \item At least one of $t$ or $t_n$ is a plug created at this refinement level.
  \item Splitting the interval created the gap $g$ that we see at this refinement level.
\end{enumerate}
In each case, $\psi_{j+1}$ maintains the desired properties:
\begin{enumerate}
  \item $\psi_{j+1}$ does not differ from $\psi_j$ on $g$, and thus by the inductive hypothesis, $\psi_{j+1}$ maintains the desired properties.
  \item From the construction of our linear system, we have that $\psi_{j+1}(t) < \psi_{j+1}(t_n)$, enforcing monotonicity, and the size of the gap was chosen so that on $g$, we have the slope set to $m = 1 - 2^{-j-1}.$
  \item Since $\rho > 0$ and $\psi_j$ is monotonic increasing, $\eta > 0$, so $\psi_{j+1}(t) < \psi_{j+1}(t_n).$ The value $\eta$ was chosen so that the slope on $g$ is $$m = \frac{\eta}{2 \rho} \le \frac{1}{2}  < 1 - 2^{-j-1}.$$
\end{enumerate}
\end{proof}

Since the uniform limit of a continuous function is continuous, and the slope of the limit is bounded,  $\psi$ is Lipschitz continuous with constant 1.
\end{proof}

\section{Results}

We implemented Algorithm~\ref{alg:Lip} in Python 2.7 using the \texttt{mpmath} package for extended precision accuracy. We used the package \texttt{intervaltree} to provide an interval tree data structure, to efficiently store the system of towns. This code was executed (in serial) on a Razer Blade computer with an i7 processor. The families of towns in Fig.~\ref{fig:towns} are produced during the first four iterations.
\begin{figure}[p]
    \hspace*{\fill}%
  \subcaptionbox{$j = 1$}{\includegraphics[width=\textwidth, height=0.2\textheight]{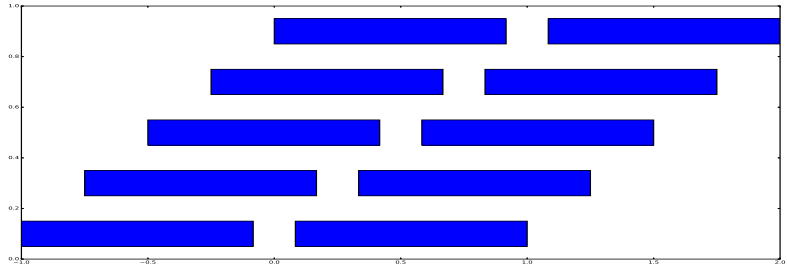}}
  \hspace*{\fill} \\
  \hspace*{\fill}%
  \subcaptionbox{$j = 2$}{\includegraphics[width=\textwidth, height=0.2\textheight]{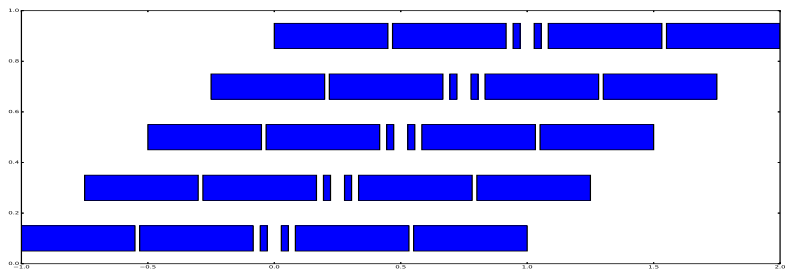}}%
    \hspace*{\fill} \\
    \hspace*{\fill}%
  \subcaptionbox{$j = 3$}{\includegraphics[width=\textwidth, height=0.2\textheight]{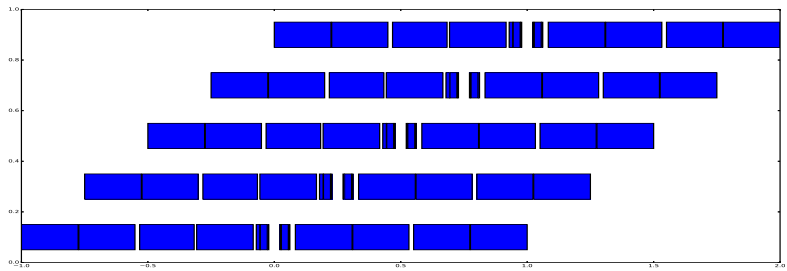}}%
  \hspace*{\fill} \\
  \hspace*{\fill}%
  \subcaptionbox{$j = 4$}{\includegraphics[width=\textwidth, height=0.2\textheight]{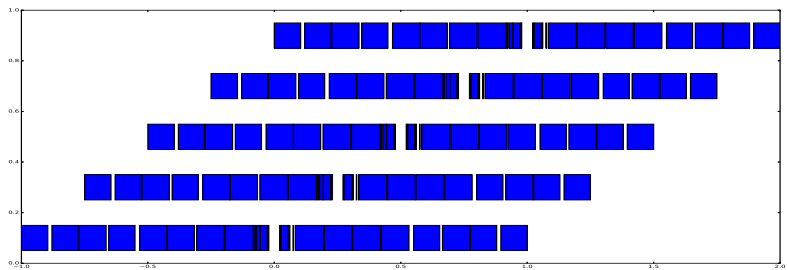}}%
    \hspace*{\fill}%
\caption{Families of towns, with shifts shown, after the first four iterations of refinement.}
\label{fig:towns}
\end{figure}

Taking a closer look at these families, we observe that there is at most one family that has a gap for any point in $[0,1]$; this is verified for iterations $j=3,\,4$ in Fig.~\ref{fig:gapcheck}. As in Fig.~\ref{fig:towns}, we can clearly see the lengths of the largest intervals are approximately halved between iterations, even if the gaps created are small.
\begin{figure}[tb]
  \hspace*{\fill}%
\subcaptionbox{$j = 3$}{\includegraphics[width=\textwidth, height=0.2\textheight]{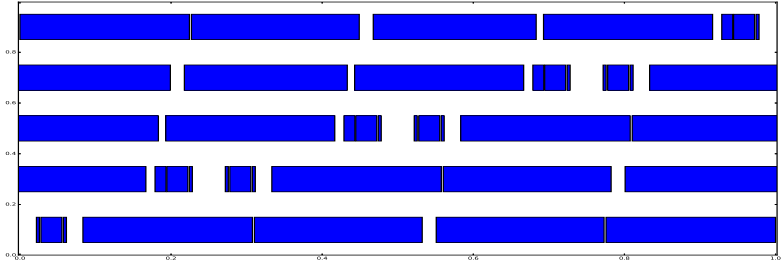}}%
\hspace*{\fill} \\
\hspace*{\fill}%
\subcaptionbox{$j = 4$}{\includegraphics[width=\textwidth, height=0.2\textheight]{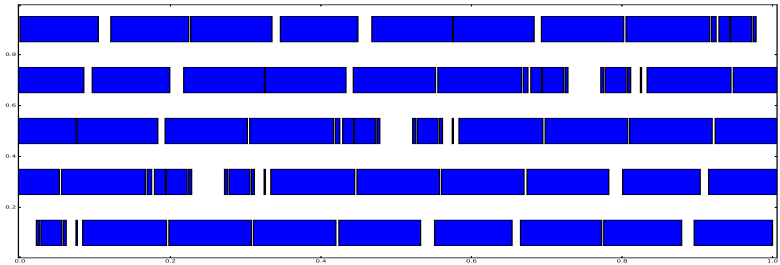}}%
  \hspace*{\fill}%
\caption{Families of towns from refinement levels $j=3,\,4$, focusing on the interval $[0,1]$.}
\label{fig:gapcheck}
\end{figure}
We produce an inner function $\psi$ as in Fig.~\ref{fig:psi}.

\begin{figure}[tb]
  \includegraphics[width=\textwidth]{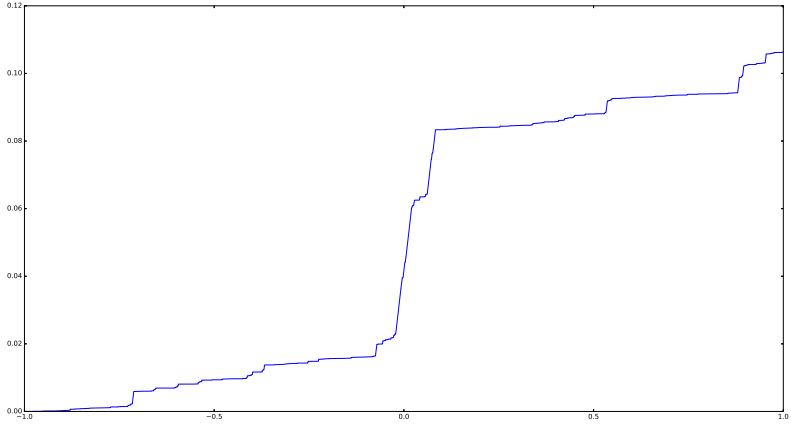}
  \caption{Function $\psi$ generated after 11 iterations.}
  \label{fig:psi}
  \end{figure}

\section{Discussion}

This algorithm improves upon the smoothness of previous versions of inner KST functions, such as the H\"older continuous
versions proposed by K\"oppen~\cite{koppen} and Braun and Griebel~\cite{braun-griebel}. It has been
argued~\cite{griebel-highdim} that the functions underlying the Kolmogorov representation lack sufficient regularity for
efficient resolution of functions for real-world applications. However, the Lipschitz regularity of our representation,
which we will extend to the outer functions in an upcoming publication, should answer these concerns. We envision that
this constructive version of the Fridman $\psi$ function, paired with an efficient method of constructing for the outer
functions $\chi^q$, will enable us to practically compute the Kolmogorov representation of multivariate functions. Our
construction opens the door to a wide variety of applications, such as encryption~\cite{Liu2015}, content-based image
and video retrieval~\cite{Bryant2008}, and image compression~\cite{LeniFougerolleTruchetet2010}.

\appendix
\section*{Appendix: Proof of Kolmogorov Superposition Theorem}
We conclude the rest of the proof of the Kolmogorov Superposition Theorem.
\\

Define by induction the functions $\chi^q$ as $$\chi^q = \lim_{r \rightarrow \infty} \chi^q_r,$$ with an initialization $\chi^q_0 \equiv 0$.
Induction on $r$ relates to refinement level $j_r \in \N$. We define $$f_r(x) = \sum_{q=0}^{2n} \chi^q_r \left( \Psi^q(x) \right), \tab M_r = \norm{f - f_r}_\infty.$$

For the base case $r = 0$, we get $f_0 \equiv 0$ and $M_0 = \norm{f}_\infty$. Suppose we constructed the continuous function $\chi^q_{r-1}$ by induction, having defined some $j_{r-1} \in \N$ and a continuous function $f_{r-1}$. Choose $k_r$ such that the oscillation of $f - f_{r-1}$ is bounded by $\frac{1}{2n+2}M_{r-1}$ on any specific $S^q_{j_r;i_1,\dots,i_n}$. \\

Fix $\xi^q_{j_r} = (\xi^q_{j_r,i_1}, \dots, \xi^q_{j_r,i_n}) \in S^q_{j_r;i_1,\dots,i_n}$. For $y \in \Psi^q(S^q_{j_r;i_1,\dots,i_n})$, define
$$\chi^q_r(y) = \chi^q_{r-1}(y) + \frac{1}{n+1} \left[ f(\xi^q_{j_r}) - f_{r-1}(\xi^q_{j_r})\right].$$

Rearranging, this gives us $\norm{\chi^q_r- \chi^q_{r-1}}_\infty \le \frac{1}{n+1} M_{r-1}$ when restricting $\chi^q_r,\,\chi^q_{r-1}$ to  $\Psi^q(S^q_{j_r;i_1,\dots,i_n}).$ Outside of $\Psi^q(S^q_{j_r;i_1,\dots,i_n})$, we choose $\chi^q_r$ such that we maintain continuity and that $\norm{\chi^q_r- \chi^q_{r-1}}_\infty \le \frac{1}{n+1} M_{r-1} $ ; we know we can do so by the Tietze Extension Theorem.

Fix $x = (x_1,\dots,x_n) \in \I^n$. Then,
$$f(x) - f_r(x) = f(x) - f_{r-1}(x) - \sum_{q=0}^{2n} \chi^q_{r}\left( \Psi^q(x)\right) - \chi^q_{r-1}\left( \Psi^q(x) \right).$$
Let $\curly{Q}_1 = \left\{ q \in \{0,\dots,2n\} \,\,\vert\,\, x \in S^q_{j_r;i_1,\dots,i_n} \text{ for some indices }i_1,\dots,i_n \right\} \subset \{0,\dots,2n\}$,
and $\curly{Q}_2 = \{0,\dots,2n\} \backslash \curly{Q}_1$. By lemma \ref{lemma:cubes}, $\lvert \curly{Q}_1 \rvert = n+1,\,\lvert \curly{Q}_2 \rvert = n.$

For $q \in \curly{Q}_1,$
\begin{equation*}\begin{split}
\chi^q_{r}\left( \Psi^q(x) \right) - \chi^q_{r-1}\left( \Psi^q(x) \right) &= \frac{1}{n+1}\left[ f(\xi^q_{j_r}) - f_{r-1}(\xi^q_{j_r})  \right] \\
&= \frac{1}{n+1}\left[ f(x) - f_{r-1}(x)  \right] + \frac{\omega^q}{n+1},
\end{split} \end{equation*}
where $\omega^q$ relates to the oscillation of $f - f_{r-1}$ on the square $S^q_{j_r;i_1,\dots,i_n}$ that includes $x$. By above,
$$\lvert \omega^q \rvert \le \frac{1}{2n+2}M_{r-1}.$$

For $q \in \curly{Q}_2$, we recall our prior estimate $$\norm{\chi^q_r - \chi^q_{r-1}}_\infty \le \frac{1}{n+1}M_{r-1}.$$
Therefore,
\begin{equation*} \begin{split}
\norm{f - f_r}_\infty &\le  \frac{1}{n+1} \sum_{q \in \curly{Q}_1} \lvert \omega^q \rvert  +  \sum_{q \in \curly{Q}_2} \norm{\chi^q_r - \chi^q_{r-1}}  \\
&\le \frac{1}{2n+2}M_{r-1} + \frac{n}{n+1}M_{r-1} \\
&= \left(\frac{2n+1}{2n+2}\right) M_{r-1}.
\end{split} \end{equation*}
Therefore, $$M_r \le \frac{2n+1}{2n+2}M_{r-1}, \tab M_r \le \left( \frac{2n+1}{2n+2} \right)^r M_0.$$ Since for all $0 \le q \le 2n,$ we have
$$\norm{\chi^q_r - \chi^q_{r-1}}_\infty \le \frac{1}{n+1}M_{r-1},$$
we conclude that the sequence $\{\chi^q_r\}_{r \in \N}$ is a Cauchy sequence. Since the space of continuous functions on $I^n$ is complete, the function $\chi^q = \lim_{r \rightarrow \infty} \chi^q_r$ is well-defined and continuous. Since $M_r \rightarrow 0$ as $r \rightarrow \infty$, we conclude that $f_r \rightarrow f$, thus completing the proof of theorem \ref{theorem:KST}.

\section*{Acknowledgements}
JA would like to acknowledge support from the Ken Kennedy Institute Computer Science \& Engineering Enhancement
Fellowship, funded by the Rice Oil \& Gas HPC Conference. MGK would like to acknowlege partial support from NSF award
SI2-SSI: 1450339 and U.S. DOE Contract DE-AC02-06CH11357.

\bibliographystyle{siamplain}
\bibliography{KST}

\end{document}